\documentclass{article}
\usepackage{amsthm,amsfonts, amsbsy, amssymb,amsmath,graphicx}
\usepackage{graphics}
\newtheorem{thm}{Theorem}

\newtheorem{crl}{Corollary}
\title{On Groups $G_{n}^{2}$ and Coxeter Groups}

\author{Vassily Olegovich Manturov}

\begin{document}

\maketitle

\begin{abstract}
In the present paper, we prove that the group $G_{n}^{2}$ of free $n$-strand braids is isomorphic to a subgroup of a semidirect product of some Coxeter group that we denote by $C(n,2)$ and the symmetric group $S_{n}$.

\end{abstract}

In \cite{Great}, groups $G_{n}^{k}$ closely related to braid groups and various problems in geometry, topology, and dynamical systems were constructed. Many invariants of topological objects are valued in groups $G_{n}^{k}$, hence it is extremely important to solve the word problem and the conjugacy problem for $G_{n}^{k}$.

The first non-trivial instance of the groups $G_{n}^{k}$ is the group $G_{n}^{2}$; deeper understanding of $G_{n}^{2}$ may allow one to shed light to $G_{n}^{k}$ for larger $k$ since there are various homomorphisms $G_{n}^{k}\to G_{n-1}^{k}$ and
$G_{n}^{k}\to G_{n-1}^{k-1}$ finally leading to various homomorphisms $G_{n}^{k}\to G_{m}^{2}$ for various $m<n$ \cite{Great,MN}.

The aim of the present note is to prove that some finite index subgroup of the group $G_{n}^{2}$ is isomorphic to a finite index subgroup of some Coxeter groups of a graph on $\left(\begin{array}{c} n \cr \\ 2 \end{array}\right)$ vertices. 

This leads to an algebraic solution of the word problem for the groups $G_{n}^{2}$. The final result can be formulated in a form of Theorem \ref{Vinberg} that $G_{n}^{2}$ embeds into a semidirect product of some Coxeter group $C(n,2)$ and the permutation group, nevertheless, we first describe it in the term of the ``rewriting procedure'' which allows one to represent words from Coxeter groups of a special type in terms of $G_{n}^{2}$, the latter having intrinsic geometrical meaning \cite{Great,Bracket,IMN,MW}.

The question whether the groups $G_{n}^{k}$ for $k>2$ admit a similar description in terms of Coxeter groups, still remains open.

Let $\Gamma$ be the graph on $\left(\begin{array}{c} n \cr \\ 2 \end{array}\right)$ vertices whose vertices are denoted by letters $b_{ij}$, where $i,j$ run all unordered pairs of unequal integers from $\{1,\dots, n\}$. We shall connect two vertices by an edge if they share an index, thus, $b_{ij}$ is connected to $b_{jk}$ for each three distinct letters $i,j,k$.
Let us mark such edges with $3$, thus, the corresponding Coxeter group $C(n,2)$ will be given by presentation

$$\langle b_{ij}| (1),(2),(3)\rangle, $$
where the three groups of relations are
$$(b_{ij}b_{ik})^{3}=1, \forall i,j,k\in \{1,\cdots,n\}, Card(\{i,j,k\})=3, \eqno{(1)}$$
$$b_{ij}b_{kl}=b_{kl}b_{ij}, \forall i,j,k,l\in \in\{1,\cdots,n\}: Card(\{i,j,k,l\})=4 \eqno{(2)}$$
$$b_{ij}^{2}=1, \forall i,j\in \{1,\cdots, n\}, i\neq j.\eqno{(3)}$$

The groups $G_{n}^{2}$ are known as {\em free braid groups} \cite{Great};
they are also known under other names ({\em e.g., virtual Gauss braids}), see, e.g.\cite{Bardakov,BBD}.

For an integer $n>2$, we define the group $G_{n}^{2}$ as the group
having the following $\left(\begin{array}{c} n \cr \\ 2 \end{array}\right)$ generators
$a_{m}$, where $m$ runs the set of all unordered $k$-tuples
$m_{1},\dots, m_{k},$ whereas each $m_{i}$ are pairwise distinct
numbers from $\left\{1,\dots, n\right\}$.

$$G_{n}^{k}=\langle a_{m}|(1'),(2'),(3')\rangle.$$

Here the defining relations look as follows:
$$a_{ij}a_{ik}a_{jk}=a_{jk}a_{ik}a_{ij}  \eqno{(1')}$$
for each three distinct indices $i,j,k\in \{1,\cdots, n\}$.

$$a_{m}a_{m'}=a_{m'}a_{m},\eqno{(2')}$$
where $m$ and $m'$ are two disjoint $2$-element subsets of $\{1,\cdots, n\}$.

Finally, for any distinct $i,j\in \{1,\cdots, n\}$ we set
$$a_{ij}^{2}=1. \eqno{(3')}$$

Let $l:G_{n}^{2}\to S_{n}$ be the homomorphism from $G_{n}^{2}$ to the symmetric group
taking each $a_{ij}$ to the transposition $(i,j)$. 

Note that there is a similar homomorphism $m:C(n,2)\to S(n)$ which just takes $b_{ij}$ to the transposition $(i,j)$. As we shall see later, these homomorphisms are closely related to each other.

We say that an element $\beta$ of $G_{n}^{2}$ is {\em pure} if $l(\beta)=1$. Hence, we get a normal subgroup $PG_{n}^{2}=Ker(l)$ of the group $G_{n}^{2}$.
This normal subgroup is similar to the subgroup of pure braids in the group of all braids.

Let $w=a_{i_{1,1},i_{1,2}},\cdots, a_{i_{k,1},i_{k,2}}$ be a word representing an element from $G_{n}^{2}$. For $j=1,\cdots, k$ by $w_{j}$ we denote the product of the first $j$ letters of $w$, $w_{0}$ being the empty word. For $p=1,\cdots, k$, we define the permutation
$\sigma_{p}= l(w_{p})^{-1}$ with  $\sigma_{0}=id$. In other words, $\sigma_{0}$ is the identical permutation, and each consequent permutation  $\sigma_{j+1}$ is obtained by multiplying $\sigma_{j}$ with the transposition corresponding to $w_{j}$ on the left.

Now we set 

${\tilde w}=b_{\sigma_{0}(i_{1,1}),\sigma_{0}(i_{1,2})}b_{\sigma_{1}(i_{2,1}),\sigma_{1}(i_{2,2})}\cdots b_{\sigma_{k-1}(i_{k,1}),\sigma_{k-1}(i_{k,2})}$.

This rewriting rule is related to the following two approaches of strand enumeration {\em the local one} and {\em the global one}.

Consider the permutation group $S_{3}=G_{3}^{2}$.
View Fig. \ref{redraw}.

\begin{figure}
\centering\includegraphics[width=200pt]{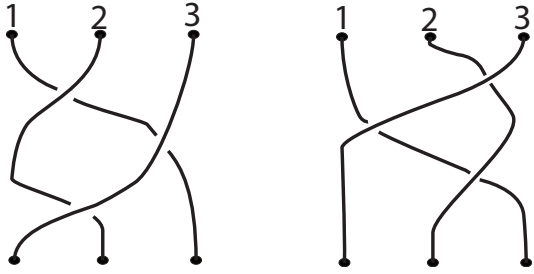}
\caption{Local enumeration yields $\sigma_{1}\sigma_{2}\sigma_{2}=\sigma_{2}\sigma_{1}\sigma_{2}$; global enumeration yields $a_{12}a_{13}a_{23}=a_{23}a_{13}a_{12}$}
\label{redraw}
\end{figure}

If we denote the crossings in a standard way (use Artin's notation), we get the standard Artin's relation $\sigma_{1}\sigma_{2}\sigma_{1}=\sigma_{2}\sigma_{1}\sigma_{2}$; if we say that $i$-th crossing $\sigma_{i}$ is formed by the intersection of $i$-th and $(i+1)$-th strands, we can even rewrite it as $b_{12}b_{23}b_{12}=b_{23}b_{12}b_{23}$ taking $b_{i,i+1}$ for $\sigma_{i}$.

However, if we enumerate the strands according to their upper ends and use $a_{ij}$ for the crossing with strands $i$ and $j$, we get the standard relation $a_{12}a_{13}a_{23}=a_{23}a_{13}a_{12}$ for $G_{n}^{2}$.

{\bf Example.} Consider the case $n=3, w=a_{12}a_{13}a_{23}a_{13}a_{23}$. Then we have the following permuations:
$\sigma_{0}=id, \sigma_{1}=(12),\sigma_{2}=(132),\sigma_{3}=id, \sigma_{4}=(13),\sigma_{5}=(213)$.
Thus, we get the corresponding word ${\tilde w}=b_{12}b_{23}b_{12}b_{13}b_{12}$, see Fig. \ref{examp}.

\begin{figure}
\centering\includegraphics[width=170pt]{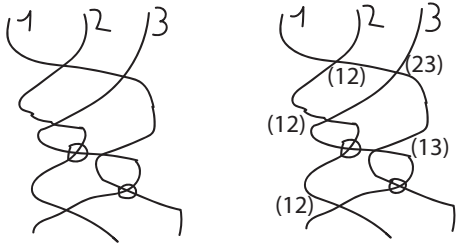}
\caption{An example of rewriting}
\label{examp}
\end{figure}

\begin{thm}
If two words $w,w'$ generate two equal elements of $G_{n}^{2}$ then the words ${\tilde w},{\tilde w}'$ generated equal elements of $C(n,2)$.

Moreover, $l(w)=m({\tilde w})^{-1}$.
 \label{thpr}
\end{thm}

\begin{proof}
It suffices for us to check the following three cases: when $w'$ is obtained from $w$ by an addition/removal of a pair of equal subsequent letters, when $w'$ is obtained from $w$ by applying the far commutativity relation  $a_{ij}a_{kl}=a_{kl}a_{ij}$ (for all $i,j,k,l$ distinct) and when $w'$ is obtained from $w$ by obtaining (1').

In the first case, at some step $p$ we have some two letters  $a_{ij}a_{ij}$ and a permutation $\sigma_{p}$; thus, instead of the letter $a_{ij}$ we get the letter $b_{\sigma_{p}(i)\sigma_{p}(j)}$. However, $\sigma_{p+1}(i)=\sigma_{p}(j),\sigma_{p+1}(j)=\sigma_{p}(i)$. Besides that, $\sigma_{p+2}\equiv\sigma_{p}$. Consequently, the addition/removal of two subsequent equal letters $w\to w'$ yields an addition/removal of two subsequent equal letters  ${\tilde w}\to {\tilde w'}$.

If in $w$ at $p$-th step we have two commuting letters $a_{ij}a_{kl}$ then the corresponding letters in  $w'$ will commute as well: the value of the permutation  $\sigma$ on $k,l$ does not depend on the value of this permutation on $i,j$ and does not change when the corresponding transpostion is applied.

Let us consider the most interesting case. Assume the letters in $w$ in the positions $p+1,p+2,p+3$ are $a_{ij}a_{ik}a_{jk}$. Then we have: $\sigma_{p+1}(j)=\sigma_{p}(i),\sigma_{p+1}(i)=\sigma_{p}(j)$. Hence, $\sigma_{p+2}(i)=\sigma_{p+1}(k)=\sigma_{p}(k);\sigma_{p+2}(k)=\sigma_{p+1}(i)=\sigma_{p}(j),\sigma_{p+2}(j)=\sigma_{p+1}(j)=\sigma_{p}(i)$.
Finally, $\sigma_{p+3}(i)=\sigma_{p+2}(i)=\sigma_{p}(k);\sigma_{p+3}(j)=\sigma_{p+2}(k)=\sigma_{p}(j),\sigma_{p+3}(k)=\sigma_{p+2}(j)=\sigma_{p}(i)$.

Thus, the corresponding three letters in the word ${\tilde w}$ will look like
$b_{\sigma_{p}(i)\sigma_{p}(j)}b_{\sigma_{p}(j)\sigma_{p}(k)}b_{\sigma_{p}(i)\sigma_{p}(j)}$.

In the same way, we check that in the word $w'$ the  images of the letters $a_{jk}a_{ik}a_{ij}$ will look like
$b_{\sigma_{p}(j)\sigma_{p}(k)}b_{\sigma_{p}(i),\sigma_{p}(j)}b_{\sigma_{p}(j)\sigma_{p}(k)}$ and the permutation after these three letters will be the same:

$\sigma_{p+3}(i)=\sigma_{p}(k);\sigma_{p+3}(j)=\sigma_{p}(j),\sigma_{p+3}(k)=\sigma_{p}(i)$.

Since the relation
$
b_{\sigma_{p}(i)\sigma_{p}(j)}b_{\sigma_{p}(j)\sigma_{p}(k)}b_{\sigma_{p}(i)\sigma_{p}(j)}=
b_{\sigma_{p}(j)\sigma_{p}(k)}b_{\sigma_{p}(i),\sigma_{p}(j)}b_{\sigma_{p}(j)\sigma_{p}(k)}
$
holds in $C(n,2)$, we get the desired statement.

The last statement of the theorem can be reformulated as follows: the product of transpositions corresponding to the word $w$ read from the right to the left corresponds to the product of transpositions corresponding to ${\tilde w}$ read from the left to the right.

This statement can be proved by induction: we start with the transposition $(i_{1,1},i_{j,1})$ for both $w$ and ${\tilde w}$
and then note that $(ab)^{-1}= a^{-1} (a b^{-1} a^{-1})$. 

\end{proof}

Denote the map $PG_{n}^{2}\to C(n,2): w\to {\tilde w}$, we have constructed, by  $co$.

It is well known that for Coxeter groups there exist a gradient descent algorithm which can means the following. Given a Coxeter group $W$ with a standard system of generators $S=\{s_{i}\}$ and relations $s_{i}^{2}=1$, $(s_{i}s_{j})^{m_{ij}}=1$.
Rewrite the latter relations as $s_{i}s_{j}\cdots = s_{j}s_{i}\cdots $ ($m_{ij}$ factors on the LHS and $m_{ij}$ factors on the RHS). This exchange relation is an  {\em elementary equivalence} which does not change the length. Besides that we have the two the elementary equivalence  $s_{j}s_{j}=1$ which decreases (increases) the length by $2$.

We say that a word $w$ in generators $s_{i}$ is {\em reduced} if its length is minimal among all words representing the same element of $W$. It is known that from each word $w$ one can get to any reduced word $w'$ representing the same element of $W$ by using only elementary equivalences not increasing the length. In particular, any for any two reduced words $w,w'$ representing
the same element of $W$, we can get from $w$ to $w'$ by a sequence of exchanges.

From
Theorem \ref{thpr} immediately see that the same is true for $G_{n}^{2}$, namely, we get the following
\begin{crl}
For each word $w$ in $a_{ij}$ and a reduced word $w'$ equivalent to it in $G_{n}^{2}$ one can get
from $w$ to $w'$ by a sequence of far commutativity relations (2'), exchanges of the type $a_{ij}a_{ik}a_{jk}\to a_{jk}a_{ik}a_{ij}$ (1'), and cancellations of two identical letters $a_{ij}a_{ij}\to \emptyset$.

In particular, if both $w$ and $w'$ are reduced then we can get from $w$ to $w'$ by using only (1') and (2').
\end{crl}

The inverse map to $co$ is constructed in a similar way: one should just take into account that if ${\tilde w}$ corresponds to $w$ then $l(w)=m({\tilde w})^{-1}$. This means that having the word ${\tilde w}$, we know all permutations $\sigma_{k}(w)$ and $l(w)$ for the potential word $w$ we are constructing.

The above arguments lead us to the following
\begin{thm}
The map $co:w\mapsto {\tilde w}$ is an isomorphism $PG_{n}^{k}\to 'C(n,2)$.
\end{thm}

\begin{proof}
Indeed, the fact that equivalent words are mapped to equivalent words is proved in Theorem \ref{thpr}. Taking into account that for words from $PG_{n}^{2}$ the permutation $l(w)$ is trivial, we see that for $w_{1},w_{2}\in PG_{n}^{2}$ the equality  ${\tilde {(w_{1}w_{2})}}={\tilde w_{1}}{\tilde w_{2}}$ takes place; hence the map $co$ is a homomorphism. The existence of the inverse homomorphism is proved analogously.
\end{proof}

The rewriting techniques can be formulated as follows. Consider the
semidirect products $G_{n}^{2}\leftthreetimes S_{n}$ and $C(n,2)\leftthreetimes S_{n}$
where the permutation group $S_{n}$ acts on the generators $a_{ij}$ (resp., $b_{ij}$) by permutations of indices.

\begin{thm}
There is an isomorphism between semidirect products $G_{n}^{2}\leftthreetimes S_{n}$ and $C(n,2)\leftthreetimes S_{n}$
which takes $(b_{ij},1)$ to $(a_{ij},(ij))$ for each transposition $(ij)$ and $(\sigma,1)\to (\sigma,1)$.

Hence, $G_{n}^{2}$ is a normal subgroup of $C(n,2)\leftthreetimes S_{n}$.
\label{Vinberg}
\end{thm}

\section{Acknowledgements}

I am grateful to L.A.Bokut' and E.B. Vinberg for useful discussions; I am especially grateful to
E.B.Vinberg who suggested the formulation of Theorem \ref{Vinberg}.

\end{document}